\documentclass{eptcs}
\providecommand{\event}{SOS 2007} % Name of the event you are submitting to

\usepackage{underscore}

\title{Negative Translations of Orthomodular Lattices and Their Logic
    }
    
 \author{Wesley Fussner
\institute{Laboratoire J.A. Dieudonn\'e, CNRS,\\ and Universit\'e C\^ote d'Azur, France\thanks{W. Fussner received funding from the European Research Council (ERC) under the European Union's Horizon 2020 research and innovation program (grant agreement No. 670624).}}
\email{wfussner@unice.fr}
\and
Gavin St.\,John
\institute{Department of Pedagogy, Psychology, Philosophy, \\ Universit\`a degli studi di Cagliari, Italy\thanks{G. St.\,John acknowledges the support of MIUR within the project PRIN 2017: ``Logic and cognition. Theory, experiments, and applications'', CUP: 2013YP4N3, sub-project ``Quantum structures and substructural logics: a unifying approach''.}}
\email{gavinstjohn@gmail.com}
}

\def\titlerunning{Negative Translations of Orthomodular Lattices and Their Logic}
\def\authorrunning{W. Fussner \& G. St.\,John}

\usepackage{amsmath,amsthm}
\usepackage{amssymb}
\usepackage{amsfonts}
\usepackage{mathtools}
\usepackage{tikz}
\usepackage{stmaryrd}
\usepackage{MnSymbol}
\usepackage{marginnote}
\usepackage{enumerate}{
\usepackage{float}
\usepackage[capitalise]{cleveref}

\numberwithin{equation}{section}

\theoremstyle{plain}
\newtheorem{theorem}{Theorem}[section]
\newtheorem{lemma}[theorem]{Lemma}
\newtheorem{proposition}[theorem]{Proposition}
\newtheorem{corollary}[theorem]{Corollary}

\theoremstyle{definition}
\newtheorem{definition}[theorem]{Definition}

\newtheorem{caution}[theorem]{Caution}

\crefname{proposition}{Prop.}{Props.}
\crefname{lemma}{Lem.}{Lems.}
\crefname{theorem}{Thm.}{Thms.}
\crefname{corollary}{Cor.}{Cors.}
\Crefname{proposition}{Proposition}{Propositions}
\Crefname{lemma}{Lemma}{Lemmas}
\Crefname{theorem}{Theorem}{Theorems}
\Crefname{corollary}{Corollary}{Corollaries}
\newcommand{\meet}{\wedge}
\newcommand{\join}{\vee}
\newcommand{\under}{\backslash}

\newcommand{\m}{\mathbf}

\newcommand{\eq}{\approx}

\newcommand{\V}{\mathsf{V}}

\newcommand{\rneg}{{\sim}}
\newcommand{\dblr}[1]{\overline{#1}} % double \rneg negation
\newcommand{\rdot}{*}
\newcommand\rol{\mathsf{ROL}}
\newcommand\W{\mathsf{W}}
\newcommand\U{\mathsf{U}}
\newcommand\vk{\vdash_{\sf K}}

\renewcommand{\aa}{1/2}
\newcommand{\bb}{3/2}

\newcommand\runder{-\hspace{-.5em}-\hspace{-.45em}\rdot}%{-\!\!\!\!\rdot}
\newcommand{\Trs}{\mathrm{T}}
\newcommand{\omlb}[1]{\dblr{#1}}
\renewcommand{\vec}[1]{\mathbf{#1}}
\newcommand{\shook}{\to} % sasaki hook
\newcommand{\cohook}{\odot}
\begin{document}
\maketitle

\begin{abstract}
We introduce \emph{residuated ortholattices} as a generalization of---and environment for the investigation of---orthomodular lattices. We establish a number of basic algebraic facts regarding these structures, characterize orthomodular lattices as those residuated ortholattices whose residual operation is term-definable in the involutive lattice signature, and demonstrate that residuated ortholattices are the equivalent algebraic semantics of an algebraizable propositional logic. We also show that orthomodular lattices may be interpreted in residuated ortholattices via a translation in the spirit of the double-negation translation of Boolean algebras into Heyting algebras, and conclude with some remarks about decidability.
\end{abstract}

\section{Introduction}
Orthomodular lattices have been studied extensively as an algebraic foundation for reasoning in quantum mechanics (see, e.g., \cite{DCGG2004}), and their assertional logic is among the most prominent quantum logics (see, e.g., \cite[p.~483]{F2016}). Regrettably, the algebraic theory of orthomodular lattices suffers from several defects that have inhibited its study. For instance, the variety of orthomodular lattices is not closed under MacNeille completions \cite{H1991} or even canonical completions \cite{H1998}. Because existing proofs that ortholattices have the finite model property invoke the MacNeille completion \cite{B1976}, this presents a significant obstacle in tackling decidability questions. Indeed, it remains an open question whether the equational theory of orthomodular lattices (or, equivalently, their assertional logic) is decidable.

Many of these challenges seem to be due to the orthomodular law itself, whose properties contribute to the underlying difficulty of the previously-mentioned questions. Consequently, one plausible approach to address these questions is to embed orthomodular lattices in an environment that is more amenable from the perspective of completions, decidability, proof theory, and related issues. The present study contributes to research in this direction, introducing \emph{residuated ortholattices} as a candidate for such an amenable environment.

\Cref{sec:algebra} defines and contextualizes residuated ortholattices, and undertakes a preliminary study of their algebraic properties. Notably, we provide several characterizations of orthomodular lattices within this environment in \Cref{sec:term def}. We subsequently establish in \Cref{sec:logic} that (in contrast to ortholattices) residuated ortholattices are the equivalent algebraic semantics of their assertional logic. Finally, in \Cref{sec:translation} we exhibit a double-negation translation of orthomodular lattices into residuated ortholattices. As an application of this translation, we show that the decidability of the equational theory of any of several varieties of residuated ortholattices suffices to guarantee the decidability of the equational theory of the corresponding variety of orthomodular lattices. In particular, if the equational theory of residuated ortholattices is decidable, then so is the equational theory of orthomodular lattices.

\section{From orthomodular lattices to residuated ortholattices}\label{sec:algebra}

We assume familiarity with lattice theory, universal algebra, and algebraic logic, and we invite the reader to consult \cite{BS1981,GJKO2007,F2016} as references on these topics. A \emph{bounded involutive lattice} is a bounded lattice $(A,\meet,\join,0,1)$ equipped with an antitone involution $\neg$. A bounded involutive lattice is called an \emph{ortholattice} (or an \emph{OL}) if it satisfies either of the equivalent identities $x\meet\neg x \eq 0$ or $x\join\neg x \eq 1$,\footnote{As usual in universal-algebraic studies, we use the symbol $\eq$ to denote formal equality.} and an ortholattice ${\m A}$ is called an \emph{orthomodular lattice} (or an \emph{OML}) if it satisfies the quasiequation
$$x\leq y \implies y \eq x\join (y\meet\neg x),$$
where as usual $x\leq y$ abbreviates $x\meet y \eq x$. Equivalently, by replacing $x$ by $x\meet y$, orthomodular lattices may be defined relative to ortholattices by the identity $y\eq (x\meet y)\join (y\meet\neg (x\meet y)).$ Due to their relevance in the logical foundations of quantum mechanics as well as purely algebraic concerns, ortholattices and orthomodular lattices are the subject of an extensive literature; see e.g. \cite{BH2000,DCGG2004} for an overview.

In any bounded involutive lattice $(A,\meet,\join,\neg,0,1)$, we may define two binary operations $\cdot$ and $\shook$ by
$$x\cdot y := x\meet (\neg x\join y),$$
$$x\shook y := \neg x\join (x\meet y),$$
for all $x,y\in A$. The operation $\cdot$ is usually called \emph{Sasaki product},\footnote{Note that some authors denote the term $x\meet (\neg x\join y)$ by $y\cdot x$, while others denote it by $x\cdot y$ as we do here.} and as usual we will often abbreviate $x\cdot y$ by $xy$. The operation $\shook $ is well-known as a candidate for an implication-like operation in OMLs (see, e.g., \cite{MP2003}), and has been called \emph{Sasaki hook} in this context. However, in our more general setting, $\shook $ will not behave as an implication. We caution that neither $\cdot$ nor $\shook $ is associative or commutative in general. 

If ${\m A} = (A,\meet,\join,\neg,0,1)$ is a bounded involutive lattice, we say that $\cdot$ is \emph{residuated}\footnote{Most studies of residuated structures focus on binary operations with both left and right (co-)residuals. Since we consider only structures with a (co-)residual on one side, we will simply use the term \emph{(co-)residuated} for brevity.} provided that there exists a binary operation $\under$ on $A$ such that for all $x,y,z\in A$,
\begin{equation}\tag{R}\label{eq:R}
x\cdot y\leq z \iff y\leq x\under z.
\end{equation}
Dually, we say that $\shook $ is \emph{co-residuated} if there exists a binary operation $\cohook$ on $A$ such that for all $x,y,z\in A$,
\begin{equation}\tag{CoR}\label{eq:CoR}
y\leq x\shook z \iff x \cohook y\leq z.
\end{equation}
Chajda and L\"anger show in \cite{CL2017} that for each orthomodular lattice ${\m A}$, Sasaki product $\cdot$ is residuated and $x\under y = x \shook  y$ for all $x,y\in A$ (and therefore also that $\shook $ is co-residuated and $x \cohook y = x\cdot y$ for all $x,y\in A$). This does not hold for bounded involutive lattices generally, but we may obtain the following.

\begin{proposition}\label{prop:res iff cores}
Let ${\m A}$ be a bounded involutive lattice. The following are equivalent.
\begin{enumerate}[\quad(1)]
\item The operation $\cdot$ is residuated.
\item The operation $\shook $ is co-residuated.
\end{enumerate}
Moreover, if ${\m A}$ is a bounded involutive lattice for which the above equivalent conditions hold, then ${\m A}$ is an ortholattice.
\end{proposition}

\begin{proof}
Suppose that $\cdot$ is residuated and that $\under$ is its residual. For $x,y\in A$ we define a binary operation $\cohook$ on $A$ by
$$x\cohook y = \neg (y\under\neg x).$$
Now observe that for all $x,y,z\in A$,
\begin{align*}
x\leq y \shook  z &\iff x\leq \neg (y\cdot\neg z)\\
&\iff y\cdot\neg z \leq \neg x\\
&\iff \neg z\leq y\under\neg x\\
&\iff \neg (y\under\neg x) \leq z.\\
&\iff x\cohook y\leq z.
\end{align*}
Thus $\cohook$ is a co-residual for $\shook $. The converse follows by a similar argument, showing that if $\cohook$ is a co-residual for $\shook $, then $x\under y = \neg (x\cohook \neg y)$ defines a residual for $\cdot$.

Now suppose that ${\m A}$ satisfies the equivalent conditions (1) and (2) and let $x\in A$. Direct computation shows that $x\cdot 0 = x\meet\neg x$ for all $x\in A$. On the other hand, $0\leq x\under 0$ implies by residuation that $x\cdot 0=0$. Thus $x\meet \neg x \eq 0$ holds in ${\m A}$, and ${\m A}$ is an ortholattice.
\end{proof}
\begin{definition}
A \emph{residuated ortholattice} (or \emph{ROL}) is an expansion of a bounded involutive lattice $(A,\meet,\join,\neg,0,1)$ by a binary operation $\under$ satisfying (\ref{eq:R}).
\end{definition}
There are many residuated ortholattices that are not OMLs. Table~\ref{tab:card} displays a computer-assisted count (up to isomorphism) of the number of OMLs and ROLs of cardinality at most $12$. By employing the usual methods of residuated structures \cite{GJKO2007}, one may show that the condition (\ref{eq:R}) may be replaced by a finite set of identities, whence residuated ortholattices form a variety. We denote the varieties of OLs, OMLs, and ROLs by $\sf OL$, $\sf OML$, and $\sf ROL$, respectively.

\begin{table}[t]\label{tab:card}
\begin{center}
\begin{tabular}{|c|ccccccccccc}
\hline
$n$ & \multicolumn{1}{c|}{\textbf{2}} & \multicolumn{1}{c|}{\textbf{3}} & \multicolumn{1}{c|}{\textbf{4}} & \multicolumn{1}{c|}{\textbf{5}} & \multicolumn{1}{c|}{\textbf{6}} & \multicolumn{1}{c|}{\textbf{7}} & \multicolumn{1}{c|}{\textbf{8}} & \multicolumn{1}{c|}{\textbf{9}} & \multicolumn{1}{c|}{\textbf{10}} & \multicolumn{1}{c|}{\textbf{11}} & \multicolumn{1}{c|}{\textbf{12}} \\ \hline
\textbf{OMLs} & 1 & 0 & 1 & 0 & 1 & 0 & 2 & 0 & 2 & 0 & 3 \\ \cline{1-1}
\textbf{ROLs} & 1 & 0 & 1 & 0 & 2 & 0 & 4 & 0 & 7 & 0 & 15 \\ \cline{1-1}
\end{tabular}
\end{center}
\caption{The number of OMLs and ROLs of cardinality $n$ up to isomorphism.}
\end{table}

\subsection{Basic properties of residuated ortholattices}
The following lemma is used throughout the sequel. Its proof is straightforward, and we omit it.
\begin{lemma}\label{SasProps}
Let ${\m A}$ be a bounded involutive lattice and let $x,y,z\in A$. Then:
\begin{enumerate}[\quad(1)]
\item If $y\leq z$, then $x\cdot y\leq x\cdot z$ and $x\shook y\leq x\shook z$.
\item $x\meet y \leq x\cdot y \leq x$.
\item $x\cdot x = x$.
\item $x\cdot 1 = 1\cdot x = x$.
\item $0\cdot x=0$ and $x\cdot 0 = x\meet \neg x$.
\item $x\shook 0 = \neg x$.
\item $x\cdot \neg x = \neg x \cdot x = x\meet \neg x$.
\item $x\cdot y=(\neg x \vee y)\cdot x$.
\end{enumerate}
If additionally ${\m A}$ is a residuated ortholattice and $S\subseteq A$, then the following hold:
\begin{enumerate}[\quad(1)]
\setcounter{enumi}{8}
\item $x(x\under y)\leq y$.
\item $x\under x = 1$.
\item If $y\leq z$, then $x\under y \leq x\under z$.
\item If $\bigvee S$ exists in $\m A$, then $\bigvee_{y\in S} xy$ exists in ${\m A}$ and $x\cdot \bigvee S = \bigvee_{y\in S} xy$.
\item If $\bigwedge S$ exists in $\m A$, then $\bigwedge_{y\in S} x\under y$ exists in ${\m A}$ and $x\under \bigwedge S = \bigwedge_{y\in S}x\under y$.
\end{enumerate}
\end{lemma}

Sasaki product is not generally associative, but we can establish several weak forms of associativity (compare with \cite{GGN2015}).
\begin{definition}
Let $A$ be a set and let $\star$ be binary operation on $A$. We say that $\star$ is:
\begin{enumerate}[\quad(1)]
\item \emph{left alternative} if $(x\star x)\star y = x\star (x\star y)$ for all $x,y\in A$.
\item \emph{right alternative} if $y\star (x\star x) = (y\star x)\star x$ for all $x,y\in A$.
\item \emph{alternative} if it is both left and right alternative.
\item \emph{flexible} if $(x\star y)\star x = x\star (y\star x)$ for all $x,y\in A$.
\item \emph{power associative} if $\star$ is associative in every 1-generated subalgebra of $(A,\star)$.
\end{enumerate}
\end{definition}

\begin{lemma}\label{prop: idem}
%\marginpar{G: I don't think ``bounded'' is used}
Let ${\m A}$ be a bounded involutive lattice. Then:
\begin{enumerate}[\quad(1)]
\item $\cdot$ is power associative and alternative.
\item $(xy)x\eq xy$.
\item $x(yx)\leq (xy)x$.
\item If $\m A$ also satisfies $x(y\join z)\eq xy\join xz$, then $\cdot$ is flexible. In particular, this holds if ${\m A}$ is a residuated ortholattice.
\end{enumerate}
\end{lemma}

\begin{proof} (1) The operation $\cdot$ is idempotent by \Cref{SasProps}(3), and idempotency entails power associativity.  Since $\cdot$ is idempotent, we need only verify $xy \eq x(xy)$ to prove left alternativity and $yx \eq (yx)x$ to prove right alternativity. Let $x,y\in A$. Now by \Cref{SasProps}(2) we obtain
$$x  y = x\meet x  y \leq x  (x  y) 
\quad
\&
\quad
(y  x)  x \leq y  x,
 $$
so it suffices to verify the reverse inequalities.

For left alternativity, observe that $x (x y) \leq x y$ if and only if $x (x y)\leq x$ and $x  (x  y)\leq \neg x \vee y$. The former conjunct holds by \Cref{SasProps}(2). For the latter, since $x  y\leq \neg x \vee y$ by definition, we have
$$x  (x  y)=x\meet(\neg x \vee x  y)\leq \neg x \vee x  y \leq \neg x\vee (\neg x\vee y)=\neg x\vee y.$$
For right alternativity, $y  x \leq (y  x)  x$ if and only if $y  x\leq y  x$ and $y  x \leq \neg(y  x)\vee x$. We need only verify the latter inequality since $\leq$ is reflexive. Again, $y  x\leq y$ and $y  x\leq \neg y \join x$ by definition, and since $\neg$ is order reversing, we obtain
$$y  x\leq \neg y \join x \leq \neg(y  x)\join x.$$
It follows that $\cdot$ is both left and right alternative, and hence alternative.

(2) Let $x,y\in A$. Observe that:

$$\begin{array}{r c l l}
(xy)x &=& (x\meet (\neg x\join y))x&\mbox{Definition of $\cdot$}\\ 
&=& (x\meet (\neg x\join y)) \meet [\neg (x\meet (\neg x\join y)) \join x]&\mbox{Definition of $\cdot$}\\
&=& x\meet (\neg x\join y) \meet [\neg x\join (x\meet \neg y) \join x]&\mbox{Involutive lattice properties}\\
&=& x\meet (\neg x\join y) \meet (x\join\neg x)&\mbox{Absorption law}\\
&=& x\meet (\neg x\join y)&\mbox{Absorption law}\\
&=& xy.&\mbox{Definition of $\cdot$}\\
\end{array}$$

(3) Recall that $\cdot$ is isotone in its second coordinate and $yx\leq y$ by \Cref{SasProps}(2). Thus using (2) we have for all $x,y\in A$ that $x(yx)\leq xy = (xy)x$.

(4) By (2) and (3) it is enough to verify $xy\leq x(yx)=x\meet (\neg x \join yx)$. Since $xy\leq x$ by \Cref{SasProps}(2), this is equivalent to showing $xy\leq \neg x \join yx$. Observe that:
$$\begin{array}{r c l l} 
xy &=& x\meet (\neg x \vee y)&\mbox{Definition of $\cdot$}\\
&\leq & (x\join \neg y) \meet (\neg x \vee y) &\mbox{Lattice properties} \\
&\leq& (x\join \neg y) \cdot (\neg x\vee y)&\mbox{\Cref{SasProps}(2)}\\
&=&(x\join\neg y)(\neg x) \vee (x\join \neg y)y&\mbox{Since $x(y\join z)\eq xy\join xz$}\\
&=& (\neg (\neg x)\join\neg y)(\neg x) \vee ( \neg y\vee x)y&\mbox{Involutive lattice properties}\\
&=& (\neg x) (\neg y) \join yx&\mbox{\Cref{SasProps}(8)}\\
&\leq&\neg x \vee yx &\mbox{By \Cref{SasProps}(2)}.
\end{array}
$$
Therefore the claim is settled.
\end{proof}
Note that the hypothesis of \Cref{prop: idem}(4) cannot be dropped, even if ${\m A}$ is assumed to be an ortholattice (see the example depicted in Figure~\ref{fig:not flex}).

\begin{figure}[t]
%\begin{figure}[h]
\centering
\begin{tikzpicture}
    \node[label=right:\tiny{$1=\neg0$}] at (-\aa,\bb) {$\bullet$};
    \node[label=right:\tiny{$y$}] at (-\bb,\aa) {$\bullet$};
    \node[label=right:\tiny{$\neg x$}] at (-\aa,\aa) {$\bullet$};
    \node[label=right:\tiny{$z$}] at (\aa,\aa) {$\bullet$};
    \node[label=right:\tiny{$\neg z$}] at (-\aa,-\aa) {$\bullet$};
    \node[label=right:\tiny{$\neg y$}] at (\aa,-\aa) {$\bullet$};
    \node[label=right:\tiny{$ x$}] at (\bb,-\aa) {$\bullet$};
    \node[label=right:\tiny{$0=\neg 1$}] at (\aa,-\bb) {$\bullet$};

    \draw (-\aa,\bb) -- (-\bb,\aa);
    \draw (-\aa,\bb) -- (-\aa,\aa);
    \draw (-\aa,\bb) -- (\aa,\aa);
    
    \draw (-\aa,-\aa) -- (-\bb,\aa);
    \draw (-\aa,-\aa) -- (-\aa,\aa);
    
    \draw (\aa,-\aa) -- (\aa,\aa);
    \draw (\bb,-\aa) -- (\aa,\aa);
    
    \draw (\aa,-\bb) -- (-\aa,-\aa);
    \draw (\aa,-\bb) -- (\aa,-\aa);
    \draw (\aa,-\bb) -- (\bb,-\aa);
\end{tikzpicture}
\label{fig:not flex}
\caption{The labeled Hasse diagram of an ortholattice whose Sasaki product $\cdot$ is not flexible. E.g., $(xy)x=x\neq 0 = x(yx)$.}
\end{figure}
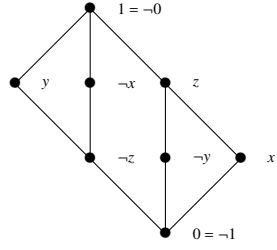

\subsection{Term-definability of residuals}\label{sec:term def}
If ${\m A}$ is an OML, the residual of $\cdot$ is given $\shook $, which is itself definable by a \emph{term} in the language $\{\meet,\join,\neg\}$. This is a remarkable property that OMLs share with Boolean algebras (but generally not other kinds of residuated structures), and has been pursued as another avenue of generalizing OMLs (see \cite{CL2020, F2020}). We will show that OMLs are the only residuated ortholattices with this property. Toward this, we recall the following well-known fact about OLs and OMLs \cite[Proposition 2.1]{BH2000}.
\begin{lemma}\label{lem:forbidden}
Let ${\m A}$ be an ortholattice, and denote by ${\m B}_6$ the ortholattice whose labeled Hasse diagram is depicted in \Cref{fig:benzene}. The following are equivalent.
\begin{enumerate}[\quad(1)]
\item ${\m A}$ is orthomodular.
\item ${\m B}_6$ is not a subalgebra of ${\m A}$.
\end{enumerate}
\end{lemma}
Using this fact, we obtain the following.
\begin{figure}[t]
%\begin{figure}[h]
\centering
\begin{tikzpicture}[baseline=(current bounding box.center)] 
    \node[label=right:\tiny{$1=\neg0$}] at (0,0) {$\bullet$};
    \node[label=left:\tiny{$a$}] at (-0.5,-0.5)  {$\bullet$};
    \node[label=left:\tiny{$\neg b$}] at (-0.5,-1)  {$\bullet$};

    \node[label=right:\tiny{$b$}] at (0.5,-0.5) {$\bullet$};
    \node[label=right:\tiny{$\neg a$}] at (0.5,-1) {$\bullet$};
    \node[label=right:\tiny{$0=\neg1$}] at (0,-1.5) {$\bullet$};

    \draw (0,0) -- (-0.5,-0.5);
    \draw (0,0) -- (0.5,-0.5);   
    \draw (-0.5,-1) -- (0,-1.5);
    \draw (0.5,-1) -- (0,-1.5);
    \draw (0.5,-0.5) -- (0.5,-1);
    \draw (-0.5,-0.5) -- (-0.5,-1);
\end{tikzpicture}
\quad\quad
\begin{tabular}{|c|cccccc}
\hline
$\under$ & \multicolumn{1}{c|}{$0$} & \multicolumn{1}{c|}{$\neg a$} & \multicolumn{1}{c|}{$\neg b$} & \multicolumn{1}{c|}{$a$} & \multicolumn{1}{c|}{$b$} & \multicolumn{1}{c|}{$1$} \\ \hline
$0$ & $1$ & $1$ & $1$ & $1$ & $1$ & $1$ \\ \cline{1-1}
$\neg a$ & $a$ & $1$ & $a$ & $a$ & $1$ & $1$ \\ \cline{1-1}
$\neg b$ & $b$ & $b$ & $1$ & $1$ & $b$ & $1$ \\ \cline{1-1}
$a$ & $b$ & $b$ & $b$ & $1$ & $b$ & $1$ \\ \cline{1-1}
$b$ & $a$ & $a$ & $a$ & $a$ & $1$ & $1$ \\ \cline{1-1}
$1$ & $0$ & $\neg a$ & $\neg b$ & $a$ & $b$ & $1$ \\ \cline{1-1}
\end{tabular}
\caption{The forbidden configuration ${\m B}_6$, also called \emph{Benzene}, that witnesses the failure of the orthomodular law.}
\label{fig:benzene}
\end{figure}
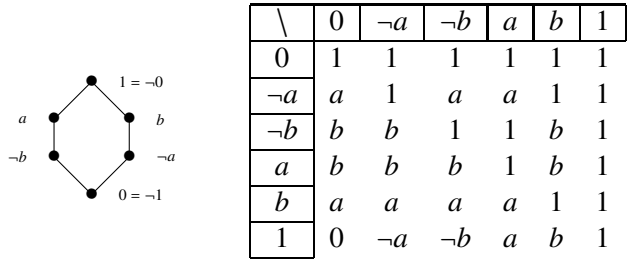
\begin{theorem}\label{thm:term definable}
Let ${\sf V}$ be a subvariety of $\sf ROL$ such that $\under$ is definable in $\sf V$ by a term in the language $\{\meet,\join,\neg,0,1\}$. Then ${\sf V}$ is a variety of OMLs.
\end{theorem}

\begin{proof}
Let $t(x,y)$ be a term in the language $\{\meet,\join,\neg,0,1\}$ such that ${\sf V}$ satisfies $t(x,y)\eq x\under y$, and toward a contradiction assume that $\sf V$ is not a variety of OMLs. Then there exists ${\m A}\in\sf V$ such that ${\m A}$ is not orthomodular, and by \Cref{lem:forbidden} we have that ${\m B}_6$ is a subalgebra of ${\m A}$ in the signature $\{\meet,\join,\neg,0,1\}$. Since $t(x,y)$ is a term in the language $\{\meet,\join,\neg,0,1\}$, we have that $t(x,y)$ defines a residual in the $\{\meet,\join,\neg,0,1\}$-subalgebra ${\m B}_6$. Because the residual of $\cdot$ is uniquely-determined when it exists, it follows that $t(x,y)$ is a term defining the operation $\under$ given in the table of \Cref{fig:benzene}. Note that every ortholattice congruence of ${\m B}_6$ respects the term $t(x,y)$, whence that every ortholattice congruence is a congruence for $\under$ as well. However, it is easy to see the ortholattice congruence generated by $(a,\neg b)$ does not respect $\under$. It follows that the residual of ${\m B}_6$ is not term-definable, a contradiction.
\end{proof}

\begin{proposition}\label{prop:residuated implies ortholattice}
Let ${\m A}$ be a bounded involutive lattice. The following are equivalent.
\begin{enumerate}[\quad(1)]
\item ${\m A}$ is an OML.
\item ${\m A}$ satisfies the quasiequation $x\leq y\implies y\cdot x\eq x$.
\item ${\m A}$ satisfies the identity $x\cdot (x\shook y)\leq y$.
\item ${\m A}$ is an ROL and ${\m A}$ satisfies the identity $x\cdot y \eq  x \odot y$, where $\odot$ is the co-residual of $\shook$.
\item ${\m A}$ is an ROL and ${\m A}$ satisfies the identity $x\under y \eq x \shook  y$.
\item ${\m A}$ is an ROL and $\under$ is definable by a term in the language $\{\meet,\join,\neg,0,1\}$.
\end{enumerate}
\end{proposition}

\begin{proof} (1) is easily seen to be equivalent to (2) from the quasiequation defining orthomodularity. If ${\m A}$ is an OML, then $x\cdot (x\shook y)\leq y$ follows because $\shook $ is an upper adjoint for $\cdot$ by \cite{CL2017}. Conversely, if (3) holds then ${\m A}$ can readily be seen to be an ortholattice since, by \Cref{SasProps},
$$x\meet \neg x = x\cdot \neg x = x\cdot(x\shook 0)\leq 0, $$
and hence $x\meet \neg x = 0$. To show that ${\m A}$ is orthomodular, suppose that $x,y\in A$ with $x\leq y$. On the one hand, $x= y\land x$ and we obtain $x = y\land x \leq y\meet (\neg y\join x) = y\cdot x$. On the other hand,  $x\leq \neg y \join x = \neg y \join (y \meet x) = y\shook  x.$ By the hypothesis and the fact that $\cdot$ preserves the order on the right, we get $ y\cdot x \leq y\cdot (y\shook x) \leq x$. Therefore $x = y\cdot x$. It follows that (1), (2), and (3) are equivalent.

(4) and (5) are readily seen to be equivalent to one another. (5) implies (3) follows because $x(x\under y)\leq y$ in any ROL, whereas the converse comes from \cite{CL2017}. Thus items (1) through (5) are equivalent, and (6) is equivalent to these as an immediate consequence of \Cref{thm:term definable} and \cite{CL2017}.
\end{proof}

\section{Congruence regularity and the logic of residuated ortholattices}\label{sec:logic}
Let $\sf K$ be a class of algebras of common similarity type $\mathcal{L}$. Recall that the \emph{relative equational consequence of ${\sf K}$} is the relation $\models_{\sf K}$ from sets of $\mathcal{L}$-equations to $\mathcal{L}$-equations defined by $E\models_{\sf K} s\eq t$ if and only if for every ${\m A}\in\sf K$ and every tuple $\vec a$ assigning elements to the variables appearing in $E\cup\{s\eq t\}$, if $u^{\m A}({\vec a})=v^{\m A}({\vec a})$ for all $(u\eq v)\in E$ then $s^{\m A}({\vec a})=t^{\m A}({\vec a})$. If further $\mathcal{L}$ contains a constant symbol $1$, the \emph{assertional logic of $\sf K$} (see \cite[Definition 3.5]{F2016}) is the logic $(\mathcal{L},\vk)$, where $\vk$ is the relation from sets of $\mathcal{L}$-formulas to $\mathcal{L}$-formulas given by
$$\Gamma\vk\varphi \iff \{\gamma\eq 1 : \gamma\in\Gamma\}\models_{\sf K}\varphi\eq 1.$$
The assertional logic of $\sf OL$ is a textbook example of a logic that is weakly algebraizable but not algebraizable \cite[Example 6.122.5]{F2016}. This defect is related to the structure of congruences of OLs. Recall that an algebra ${\m A}$ with a constant $1$ is said to be $1$-regular if for any congruences $\theta,\psi$ of $\m A$ we have that $[1]_\theta=[1]_\psi$ implies $\theta=\psi$, where $[a]_\theta$ denotes the $\theta$-congruence class of $a\in A$. A variety $\V$ whose language has a designated constant $1$ is said to be \emph{$1$-regular} if all of the algebras in $\V$ are $1$-regular. It is well-known \cite[Proposition 4.3]{BH2000} that $\sf OML$ is $1$-regular, and that the assertional logic of every $1$-pointed, $1$-regular variety is algebraizable in the sense of Blok and Pigozzi (see \cite[Theorem 6.146]{F2016} and \cite{BP1989}).

\begin{lemma}\label{lem:regularity}
The variety $\rol$ of residuated ortholattices is $1$-regular.
\end{lemma}

\begin{proof}
Let ${\m A} = (A,\meet,\join,\neg,\under,0,1)$ be a residuated ortholattice and suppose that $\theta,\psi$ are congruences of ${\m A}$ such that $[1]_\theta= [1]_\psi$. Let $(x,y)\in\theta$. Then by applying the fact that $\theta$ is a congruence for $\under$ and $(x,x),(y,y)\in\theta$, we have that $(x\under x,x\under y),(y\under y,y\under x)\in\theta$, i.e., $(1,x\under y),(1,y\under x)\in\theta$. It follows from the hypothesis that $(1,x\under y), (1,y\under x)\in\psi$. Since $\psi$ is a congruence for $\meet,\join,\neg$, $\psi$ is also a congruence for $\cdot$. Hence it follows that $(x,x(x\under y)),(y,y(y\under x))\in\psi$ as well. Now because $x(x\under y)\join y = y$ and $y(y\under x)\join x = x$, it follows that
$(x\join y, y)\in\psi$ and $(x,x\join y)\in\psi$. By transitivity, we obtain that $(x,y)\in\psi$ and $\theta\subseteq\psi$. A symmetric argument shows that $\psi\subseteq\theta$, whence $\theta=\psi$. 
\end{proof}
The following is an immediate corollary of \Cref{lem:regularity} and the preceding remarks. It demonstrates that the logical deficiencies of $\sf OL$ (as compared to $\sf OML$) can be ameliorated by the expressive power afforded by adding a residual.
\begin{theorem}\label{thm:algebraizable}
The assertional logic of $\rol$ is algebraizable in the sense of Blok and Pigozzi and its equivalent algebraic semantics is $\rol$.
\end{theorem}

Among other consequences, this entails that the lattice of axiomatic extensions of the assertional logic of $\rol$ is dually isomorphic to the lattice of subvarieties of $\rol$. Although one could extract a syntactic presentation of the assertional logic of $\rol$ (e.g., providing a Hilbert-style calculus), we do not further address this issue in the present paper.

\section{A negative translation and relative decidability}\label{sec:translation}
As a final topic for this paper, we exhibit a negative translation of $\sf OML$ into $\sf ROL$ inspired by \cite{GO2006}. For this, we will need a number of preliminary lemmas.

\subsection{Preliminaries to the translation}

\begin{lemma}\label{SasRes}
Let $\m A$ be a residuated ortholattice. Then for all $x,y\in A$:
\begin{enumerate}[\quad(1)]
\item $\neg x \leq x\under y$.
\item $\neg(x\under y)\leq x \leq (\neg x)\under y$.
\item $(x\under y)x=x\meet x\under y =x(x\under y)$.
\item $x\meet x\under y \leq y$.
\item $x\under (x\meet y) = x\under y$.
\item $xy=x(yx)$.
\item $xy=0$ if and only if $yx=0$.
\end{enumerate}
\end{lemma}
\begin{proof}
Recall that $x(\neg x)=(\neg x)x=x\cdot 0=x\meet\neg x=0$ by \Cref{SasProps}. (1) and (2) are easy computations. For (3), observe that:
$$\begin{array}{r c l l}
(x\under y)x &=& x\under y \meet (\neg(x\under y)\vee x) & \mbox{Definition of $\cdot$}\\
&=&x\under y \meet x & \mbox{By (2)}\\
&=& x \meet (\neg x \vee x\under y)& \mbox{By (1)}\\
&=& x(x\under y) &\mbox{Definition of $\cdot$}
\end{array}
$$
Note that (4) follows immediately from (3) and the fact that $x(x\under y)\leq y$. 

For (5), the $\leq$ direction follows from \Cref{SasProps}(11). 
%residuation as $x(x\under(x\meet y))\leq x\meet y\leq y$. 
On the other hand, $x\under y \leq x\under (x\meet y)$ if and only if $x(x\under y)\leq x\meet y$, which holds by \Cref{SasProps}(2,9). (6) follows directly from \Cref{prop: idem}(2) and (4). Clearly, (7) follows from (6) since $yx=0$ implies $xy = x(yx)=x\cdot 0 = 0$. The converse follows symmetrically.
\end{proof}
Given a residuated ortholattice $\m A = (A,\meet,\join,\under,0,1)$, we define the following operations on $\m A$: 
$$
\begin{array}{c  c c }
\rneg x := x\under 0 && \dblr x := \rneg\rneg x\\
x \rdot y := x\meet(\rneg x \join y) && x\runder y := \rneg x\vee(x\meet y)
\end{array}
$$
Define also the sets $\rneg X=\{\rneg x: x\in X \}$ and $\dblr X = \{\dblr x: x\in X \}$ for $X\subseteq A$.

\begin{lemma}\label{lem:Rneg}
Let $\m A$ be a residuated ortholattice and let $x,y\in A$. Then $\rneg$ is antitone, and:
\begin{enumerate}[\quad(1)]
\item $\rneg 1 = 0$ and $\rneg 0=1$.
\item $x\leq \rneg\rneg x$
\item $\rneg x = \rneg\rneg \rneg x$. Hence $\rneg A \subseteq \dblr{A}$ and $\dblr{A}=\dblr{\dblr{A}}$.
\item $\rneg(x\join y)= \rneg x \meet \rneg y$. \item $\rneg \rneg x = \rneg \neg x$.
\item If $x,y\in\dblr{A}$ and $x\leq y$, then $y\rdot x = x$.
\item $\rneg x \join \rneg y = \rneg (x\meet y)$.
\end{enumerate}
\end{lemma}
\begin{proof}
As a consequence of \Cref{SasRes}, we have $\neg x \leq \rneg x$, $\neg\rneg x \leq x\leq \rneg\neg x$, and $\rneg x\cdot x=x \meet \rneg x = x\cdot \rneg x=0$. For the antitonicity of $\rneg$, suppose $x\leq y$. 
We wish to show $\rneg y\leq \rneg x$. By residuation, it is enough to show $x\cdot \rneg y\leq 0$, or equivalently (by \Cref{SasRes}(7)), that $\rneg y\cdot x\leq 0$. Since $\cdot$ preserves the order in its right coordinate, $\rneg y\cdot x\leq \rneg y \cdot y \leq 0$.  
%We wish to show $\rneg y\leq \rneg x$. By residuation, it is enough to show $x\cdot \rneg y\leq 0$, or equivalently (by \Cref{SasRes}(6)), that $x\cdot (\rneg y\cdot x)\leq 0$. Since $\cdot$ preserves the order in its right coordinate, $\rneg y\cdot x\leq \rneg y \cdot y \leq 0$, and hence $x\cdot (\rneg y\cdot x)\leq x\cdot 0 =0$. 
Therefore $\rneg$ is antitone. Note that the antitonicity of $\rneg$ immediately yields that:
\begin{equation}\tag{DM1}\label{eq:DM1}
\rneg(x\join y)\leq \rneg x \meet \rneg y,
\end{equation}
\begin{equation}\tag{DM2}\label{eq:DM2}
\rneg x \join \rneg y \leq \rneg (x\meet y).
\end{equation}
We now prove (1)--(7).

(1) By \Cref{SasProps}(4), $\rneg 1 = 1\cdot \rneg 1\leq 0$. On the other hand, $0\cdot 1 =0$ and thus $1\leq \rneg 0$.

(2) By residuation, $x\leq \rneg\rneg x =(\rneg x)\under 0$ if and only if $\rneg x \cdot x\leq 0$, which holds as noted above.

(3) Applying $\rneg$ to (2), $\rneg \dblr x \leq \rneg x$. On the other hand, $\rneg x\leq \dblr{\rneg x}$ by (2). Since $\rneg\rneg\rneg x =\rneg\dblr x = \dblr{\rneg x}$, the first claim follows. The second claim follows since $\rneg x = \rneg \dblr x$, and thus $\dblr x = \dblr{\dblr{x}}$.

(4) Using residuation and \Cref{SasRes}(7), we have $\rneg x \meet \rneg y\leq \rneg(x\join y)$ if and only if $(x\join y)(\rneg x \meet \rneg y)\leq 0$ if and only if $(\rneg x\vee \rneg y)(x\join y)\leq 0$ if and only if $x\vee y \leq \rneg(\rneg x\meet \rneg y)$.
%Observe that $\rneg x \meet \rneg y\leq \rneg(x\join y)$ if and only if $(x\join y)(\rneg x \meet \rneg y)\leq 0$. By \Cref{SasProps}(8), we have:
%$$\begin{array}{r c l}
%(x\join y)(\rneg x \meet \rneg y) &=&  [\neg (x\join y) \join (\rneg x \meet \rneg y)](x\join y)\\
%& =& [(\neg x\meet \neg y) \join (\rneg x \meet \rneg y)](x\join y)\\
%&=&(\rneg x\vee \rneg y)(x\join y).
%\end{array}$$ 
%Thus $(x\join y)(\rneg x \meet \rneg y)\leq 0$ iff $(\rneg x \meet \rneg y)(x\join y)\leq 0$ iff $x\vee y \leq \rneg (\rneg x \land \rneg y)$. 
Observe that:
$$\begin{array}{r c l l}
x\vee y &\leq& \rneg\rneg x \vee \rneg\rneg y& \mbox{By (2)}\\
&\leq& \rneg(\rneg x\meet \rneg y) & \mbox{By (\ref{eq:DM2})}. 
\end{array}$$
The claim then follows from (\ref{eq:DM1}).

(5) Since $\neg x\leq\rneg x$ by \Cref{SasRes}(1), the antitonicity of $\rneg$ gives $\rneg\rneg x \leq \rneg \neg x$. Thus we need only verify $\rneg \neg x \leq \rneg\rneg x$, or equivalently that $\rneg x \cdot \rneg \neg x\leq 0$. Observe that:
$$\begin{array}{r c l l}
\rneg x\cdot \rneg \neg x &=& \rneg x\meet (\neg\rneg x \join \rneg \neg x)&\mbox{Definition of $\cdot$}\\
&=&\rneg x\meet \rneg \neg x&\mbox{By \Cref{SasRes}(2)}\\
&=&\rneg ( x \join \neg x)&\mbox{By (4)}\\
&=&\rneg 1\\
&=& 0&\mbox{By (1)}.
\end{array}$$

(6) Clearly $a\cdot b \leq a\rdot b$ since $\neg a \leq \rneg a$. Hence $x\leq y$ implies $x=y \meet x\leq y\cdot x \leq y\rdot x$. Thus, it suffices to show $y\rdot x \leq x$. Since $x=\rneg\rneg x$ by assumption, this is equivalent to showing that $\rneg x \cdot (y\rdot x)\leq 0$. Note that $\neg \rneg x\leq x \leq y\rdot x $ by \Cref{SasRes}(2) and $x\leq y$, and observe:
$$\begin{array}{r c l l}
\rneg x \cdot (y\rdot x) &=& \rneg x\meet (\neg \rneg x \join y\rdot x)\\
&=&\rneg x \meet y\rdot x \\
&=& \rneg x \meet (y\meet (\rneg y \join x))\\
&=& (\rneg x \meet y) \meet (x \join \rneg y)\\
&=& (\rneg x \meet \rneg(\rneg y))\meet (x\join \rneg y)&\mbox{Since $y=\dblr y$}\\
&=& \rneg(x \join \rneg y) \meet (x\join \rneg y)&\mbox{By (4)}\\
&=& 0.
\end{array} $$

(7) Let $a= \rneg (x\meet y)$ and $b=\rneg x \vee \rneg y$. Since $b\leq a$ by (\ref{eq:DM2}), it is enough to verify $a\leq b$. We claim first that $ab=a$. Since $ab=a\meet (\neg a \join b)$, it suffices to show $\neg a \join b=1$, or equivalently, $a\meet \neg b \leq 0 $. Now, 
$$\neg b = \neg(\rneg x\vee \rneg y)=\neg \rneg x \meet \neg \rneg y \leq x\meet y,$$
by \Cref{SasRes}(2), and hence $a\meet \neg b \leq a\meet (x\meet y)=0$ since ${\m A}$ is an ortholattice. Now note that $\rneg x, \rneg y, $ and $a=\rneg (x\meet y)$ are contained in $\dblr A$ by (3). Furthermore, $\rneg x,\rneg y\leq a$ since $b\leq a$. Hence for $c\in \{\rneg x,\rneg y\}$, by (6) it follows that $a\rdot c=c$ and thus $a\cdot c\leq c$. Therefore,
$a =a(\rneg x \vee \rneg y)=a(\rneg x) \vee a(\rneg y) \leq \rneg x \vee \rneg y. $
\end{proof}

\begin{lemma}\label{cor:OmlRol}
Let ${\m A}$ be a residuated ortholattice. Then the following are equivalent:
\begin{enumerate}[\quad(1)]
\item $\m A$ is an OML.
\item $\m A$ satisfies $\rneg x \eq\neg x$.
\item ${\m A}$ satisfies $x\eq\rneg\rneg x$.
\end{enumerate}
\end{lemma}
\begin{proof}

First we show (2) and (3) are equivalent. If (2) holds then $\rneg\rneg x=\neg \neg x = x$. If (3) holds then $\neg x = \rneg\rneg(\neg x) = \rneg (\rneg \neg x)=\rneg(\rneg\rneg x)=\rneg x$ by \Cref{lem:Rneg}(5) and (3).

Now we show (1) is equivalent to (2) and (3). Supposing $\m A$ is an OML, by \Cref{prop:residuated implies ortholattice}, $\under$ and $\shook $ coincide. By \Cref{SasProps}(6), $\neg x = x\shook  0 = x\under 0 = \rneg x$. 
On the other hand, supposing (2) and (3) hold, ${\neg}$ and ${\rneg}$ coincide and $A=\dblr A$. Suppose $x,y\in A=\dblr{A}$ with $x\leq y$. Then by \Cref{lem:Rneg}(6) we have $y\rdot x = x$.  But we have $y\rdot x = y\meet (\rneg y\join x)=y\meet(\neg y\join x)=y\cdot x$, so $y\cdot x=x$. Then ${\m A}$ is an OML by \Cref{prop:residuated implies ortholattice}(2), completing the proof.
\end{proof}

\begin{lemma}\label{lem:dblr}
Let $\m A$ be a residuated ortholattice. Then for all $x,y\in A$:
\begin{enumerate}[(1)]
\item $\dblr{\dblr x}=\dblr x$ and $\dblr{\neg x}=\rneg \dblr x=\rneg x$.
\item $\dblr{x\join y}=\dblr{x}\join \dblr{y}$ and $\dblr{x\meet y}=\dblr{x}\meet \dblr{y}$.
\item $\dblr{x\cdot y}=\dblr x\rdot \dblr y$.
\end{enumerate}
\end{lemma}
\begin{proof} Clearly (1) follows from \Cref{lem:Rneg}(3,5) and (2) follows from \Cref{lem:Rneg}(4,7).
Using these facts, observe that: $\dblr{x\cdot y}=\dblr{x\meet (\neg x \join y)} =\dblr{x}\meet (\dblr{\neg x} \join \dblr{y})=\dblr x\meet(\rneg \dblr x\join \dblr y)=\dblr x\rdot \dblr y.$
\end{proof}

For a residuated ortholattice ${\m A}$, define ${\dblr{\m A}} = (\bar{A},\meet,\join,\rneg,\runder,0,1)$.

\begin{lemma}\label{omlA}
Let $\m A = (A,\meet,\join,\neg,\under,0,1)$ be a residuated ortholattice.
\begin{enumerate}[(1)]
\item $\dblr{\m A}$ is an OML.
\item The map $x\mapsto\dblr{x}$ is an ortholattice homomorphism of ${\m A}$ onto $\dblr{\m A}$.
\item $\dblr{x}\under\dblr{y}=\dblr x\runder \dblr y$ for all $x,y\in A$.
\end{enumerate}
\end{lemma}
\begin{proof}
For (1), note by \Cref{lem:Rneg}(1) and \Cref{lem:dblr}(2), it follows that $(\bar{A},\meet,\join,0,1)$ is hereditarily a bounded lattice. By \Cref{lem:Rneg} and \Cref{lem:dblr}(1), $\rneg$ is an antitone involution on $\dblr{\m A}$, which furthermore satisfies $x\meet \rneg x\approx 0$ by \Cref{SasRes}. Thus $\dblr{\m A}$ is an ortholattice. Noting that the Sasaki product in $\dblr{\m A}$ is $\rdot$, observe that $\dblr{\m A}$ satisfies the quasiequation $x\leq y\implies y\rdot x \approx x$ by \Cref{lem:Rneg}(6), whence by \Cref{prop:residuated implies ortholattice}(2) we have that $\dblr{\m A}$ is orthomodular. (2) is immediate from \Cref{lem:dblr}, and (3) is a straightforward computation using the fact that $\runder$ is the residual of $\rdot$ in the OML $\dblr{\m A}$.
\end{proof}

\begin{caution}
The identity $\rneg\rneg (x\under y) \approx \rneg\rneg x \under \rneg\rneg y$ is false in ${\m B}_6$ since $\dblr{a\under \neg b}=\dblr{b}=b\neq 1 = a\under a =\dblr{a}\under \dblr{\neg b}$. Thus the map $x\mapsto\bar{x}$ is not an ROL homomorphism.
\end{caution}

\subsection{The negative translation}

Let $t$ be a residuated ortholattice term, and recall that $\rneg t:= t\under 0$ and $\dblr t:=\rneg\rneg t$. We define the term $\Trs(t)$ inductively on the complexity of $t$ as follows: $\Trs(0)=0$, $\Trs(1)=1$, $\Trs(x)=\dblr x$ for all variables $x$, $\Trs(\neg s)=\rneg \Trs(s)$, and $\Trs(r\star s)= \Trs(r)\star \Trs(s)$ for each $\star\in\{\meet,\join,\under \}$. If $E$ is a set of equations, we define the translation of this set to be $\Trs[E] = \{\Trs(u)\eq\Trs(v): (u\eq v)\in E \}$. 

\begin{definition}
For subvarieties $\W,\V$ of $\rol$, we say the $\V$ is {\em translatable into $\W$} if for any sets of equations $E\cup\{s\eq t\}$ in the language of residuated ortholattices,
$E\models_\V {s\eq t} \iff \Trs[E]\models_\W \Trs(s)\eq \Trs(t). $ 
\end{definition}

The following is evident:

\begin{proposition}\label{decid}
For varieties $\W,\V$ of residuated ortholattices, if $\V$ is translatable into $\W$ then deciding equations in $\V$ is no harder than deciding equations in $\W$. In particular, if the equational theory of $\W$ is decidable then the same holds for $\V$.
\end{proposition}

For $\m A\in \rol$ and an $\rol$-term $t$ in $n$-variables, by $t^{\m A}\colon A^n\to A$ we mean the term function of $t$ on $\m A$. If $t$ is a unary term and $\vec a = (a_1,\ldots,a_n)\in A^n$, by $t^{\m A}(\vec a)$ we denote the tuple $(t^{\m A}(a_1),\ldots ,t^{\m A}(a_n))$, in particular we will write $\dblr{\vec a}$ as an abbreviation for $(\dblr x)^{\m A}(\vec a)$, i.e., $\dblr{\vec a}=(\dblr{a_1},...,\dblr{a_n})\in (\omlb{A})^n$.

\begin{lemma}\label{lem:gammaterms}
Let $t$ be a residuated ortholattice term, $\m A\in \rol$, and 
$\vec a$ an element of an appropriate power of $A$. Then
$\Trs(t)^{\m A}(\vec a) = t^{\omlb{\m A}}(\dblr{\vec a})$.
\end{lemma}
\begin{proof}
We proceed by induction on the complexity of $t$. Observe $\Trs(0)^{\m A} = 0^{\m A}=0^{\omlb{\m A}}$ and  $\Trs(1)^{\m A} = 1^{\m A}=1^{\omlb{\m A}}$ by \Cref{omlA}. If $t$ is a variable $x$, then $\Trs(x)^{\m A}(a)=(\dblr x)^{\m A}(a)=\dblr a = x^{\omlb{\m A}}(\dblr a)$ by definition. 

Now suppose the claim holds for terms $r$ and $s$. If $t=r\star s$ where $\star\in \{\meet,\join,\under \},$ then
$$\begin{array}{ r c l l}
\Trs(r\star s)^{\m A}(\vec a)
& =& [\Trs(r)\star \Trs(s)]^{\m A}(\vec a)&\mbox{Def. of $\Trs(-)$}\\
& =& \Trs(r)^{\m A}(\vec a)\star^{\m A} \Trs(s)^{\m A}(\vec a)&\\
&=& r^{\omlb{\m A}}(\dblr{\vec a}) \star^{{\m A}} s^{\omlb{\m A}}(\dblr{\vec a}) &\mbox{Inductive hypothesis}\\
&=& r^{\omlb{\m A}}(\dblr{\vec a}) \star^{\omlb{\m A}} s^{\omlb{\m A}}(\dblr{\vec a}) & \mbox{\Cref{lem:dblr}, \Cref{omlA}(3)}\\%
&=&(r\star s)^{\omlb{\m A}}(\dblr{\vec a}).
 \end{array}$$
Essentially the same argument establishes the case for $t=\neg s$. This completes the proof. 
\end{proof}

If $\V$ is a variety and $E$ is a set of equations in the language of $\V$, we denote the subvariety of $\V$ axiomatized by $E$ by $\V+E$.

\begin{lemma}\label{target of translation}
Let $E$ be a set of equations in the language of ROLs, and set $\V={\sf OML}+E$ and $\W=\rol+\Trs[E]$. Then:
\begin{enumerate}[\quad(1)]
\item $\V$ is a subvariety of $\W$.
\item If ${\m A}\in\W$, then $\dblr{\m A}\in\V$.
\end{enumerate}
\end{lemma}

\begin{proof}
We first prove (1). Clearly $\V$ is a subvariety of $\sf ROL$, so it suffices to show that if ${\m A}\in\V$ and $(u\eq v)\in E$, then ${\m A}$ satisfies $\Trs(u)\eq \Trs(v)$. Let $\vec a$ be a tuple in an appropriate power of $A$, and note that $\dblr{\vec a}={\vec a}$ and $\dblr{\m A} = {\m A}$ from \Cref{cor:OmlRol}. By hypothesis $u^{{\m A}}({\vec a}) = v^{{\m A}}({\vec a})$, so $u^{\dblr{\m A}}(\dblr{\vec a}) = v^{\dblr{\m A}}(\dblr{\vec a})$. It follows from \Cref{lem:gammaterms} that $\Trs(u)^{\m A}({\vec a})=\Trs(v)^{\m A}({\vec a})$, so ${\m A}$ satisfies $\Trs(u)\eq\Trs(v)$ as desired.

Now for (2), let ${\m A}\in\W$ and suppose that $(u\eq v)\in E$. Then ${\m A}$ satisfies $\Trs(u)\eq \Trs(v)$. If ${\vec a}$ is a tuple from an appropriate power of $\dblr{A}$, then as before ${\vec a} = \dblr{\vec a}$ by \Cref{cor:OmlRol}. By hypothesis we have $\Trs(u)^{\m A}({\vec a})=\Trs(v)^{\m A}({\vec a})$, and by \Cref{lem:gammaterms} we get $u^{\dblr{\m A}}({\vec a})=u^{\dblr{\m A}}(\dblr{\vec a})=v^{\dblr{\m A}}(\dblr{\vec a})=v^{\dblr{\m A}}({\vec a})$. It follows that $\dblr{\m A}$ satisfies $u\eq v$, so it follows that $\dblr{\m A}\in\V$.
\end{proof}

\begin{theorem}\label{main translation}
Let $\V={\sf OML}+E$, let $\W=\rol+\Trs[E]$, and suppose that $\U$ is a subvariety of ROLs such that $\V\subseteq\U\subseteq\W$. Then $\V$ is translatable into $\U$.
\end{theorem}

\begin{proof}
Let $D\cup \{s\eq t\}$ be a set of equations in the language of ROLs, and suppose first that $D\models_{\V} s\eq t$. Let ${\m A}\in\U$ and let $\vec a$ be a tuple of elements of the appropriate power of $A$ such that $\Trs(u)^{\m A}({\vec a}) = \Trs(v)^{\m A}({\vec a})$ holds for each equation $(u\eq v)\in D$. Then by \Cref{lem:gammaterms} we have $u^{\dblr{\m A}}(\dblr{\vec a})=v^{\dblr{\m A}}(\dblr{\vec a})$ for each $(u\eq v)\in D$. Since ${\m A}\in\U\subseteq\W$, \Cref{target of translation}(2) gives $\dblr{\m A}\in\V$. The hypothesis then implies $s^{\dblr{\m A}}(\dblr{\vec a})=t^{\dblr{\m A}}(\dblr{\vec a})$, and again applying \Cref{lem:gammaterms} yields that $\Trs(s)^{\m A}({\vec a})=\Trs(t)^{\m A}({\vec a})$. It follows that $\Trs[D]\models_{\U} T(s)\eq T(t)$ as desired.

For the converse, suppose that $\Trs[D]\models_{\U} T(s)\eq T(t)$. Let ${\m A}\in\V$, let $\vec a$ be a tuple from a suitably-chosen power of $A$, and suppose that $u^{\m A}({\vec a})=v^{\m A}({\vec a})$ for all $(u\eq v)\in D$. As before we have ${\vec a}=\dblr{\vec a}$ and ${\m A}=\dblr{\m A}$, which implies $u^{\dblr{\m A}}(\dblr{\vec a})=v^{\dblr{\m A}}(\dblr{\vec a})$. \Cref{lem:gammaterms} then gives $\Trs(u)^{\m A}({\vec a})=\Trs(v)^{\m A}({\vec a})$ for each $(u\eq v)\in D$. Since $\V\subseteq\U$ we have ${\m A}\in\U$, so the assumption gives $\Trs(s)^{\m A}({\vec a})=\Trs(t)^{\m A}({\vec a})$. A final application of \Cref{lem:gammaterms} gives $s^{\dblr{\m A}}(\dblr{\vec a})=t^{\dblr{\m A}}(\dblr{\vec a})$, so $s^{\m A}({\vec a})=t^{\m A}({\vec a})$. It follows that $D\models_{\V} s\eq t$, and this establishes that $\V$ is translatable into $\U$.
\end{proof}

In particular, we obtain the following consequence of \Cref{main translation}.

\begin{corollary}
If $\V$ is a subvariety of $\sf OML$ axiomatized relative to orthomodular lattices by a set $E$ of equations, then $\V$ is translatable into $\rol + \Trs[E]$. In particular, $\sf OML$ is translatable into $\rol$.
\end{corollary}

Specializing \Cref{decid} in light of \Cref{main translation}, we get the following.

\begin{corollary}
${\sf OML}$ has a decidable equational theory if any variety of residuated ortholattices that contains it has a decidable equational theory.
\end{corollary}

Of course, via \Cref{thm:algebraizable} these results may be exported to the assertional logics that are algebraized by the varieties mentioned above. However, we do not further pursue that line of inquiry here.

\bibliographystyle{eptcs}

\end{document}